\documentclass[11pt,reqno]{amsart}

\newtheorem{proposition}{Proposition}[section]

\newtheorem{corollary}[proposition]{Corollary}
\newtheorem{theorem}[proposition]{Theorem}

\theoremstyle{definition}
\newtheorem{definition}[proposition]{Definition}
\newtheorem{example}[proposition]{Example}
\newtheorem{examples}[proposition]{Examples}
\newtheorem{remark}[proposition]{Remark}

\newcommand{\thlabel}[1]{\label{th:#1}}
\newcommand{\thref}[1]{Theorem~\ref{th:#1}}
\newcommand{\selabel}[1]{\label{se:#1}}
\newcommand{\seref}[1]{Section~\ref{se:#1}}

\newcommand{\prlabel}[1]{\label{pr:#1}}
\newcommand{\prref}[1]{Proposition~\ref{pr:#1}}
\newcommand{\colabel}[1]{\label{co:#1}}
\newcommand{\coref}[1]{Corollary~\ref{co:#1}}
\newcommand{\relabel}[1]{\label{re:#1}}
\newcommand{\reref}[1]{Remark~\ref{re:#1}}
\newcommand{\exlabel}[1]{\label{ex:#1}}
\newcommand{\exref}[1]{Example~\ref{ex:#1}}
\newcommand{\delabel}[1]{\label{de:#1}}
\newcommand{\deref}[1]{Definition~\ref{de:#1}}
\newcommand{\eqlabel}[1]{\label{eq:#1}}
\newcommand{\equref}[1]{(\ref{eq:#1})}

\newcommand{\Cc}{\mathcal{C}}

\def\*C{{}^*\hspace*{-1pt}{\Cc}}
\def\text#1{{\rm {\rm #1}}}

\input xy
\xyoption {all} \CompileMatrices

\usepackage{amssymb}
\usepackage{color,amssymb,graphicx,amscd}

\begin{document}

%\title[Bicrossed descent theory]
%{Bicrossed descent theory and the classification of finite groups}

\title[Classifying complements for groups]
{Classifying complements for groups. Applications}

\author{A. L. Agore}
\address{Faculty of Engineering, Vrije Universiteit Brussel, Pleinlaan 2, B-1050 Brussels,
Belgium}\address{Permanent address: Department of Applied
Mathematics, Bucharest University of Economic Studies, Piata
Romana 6, RO-010374 Bucharest 1, Romania}
\email{ana.agore@vub.ac.be and ana.agore@gmail.com}

\author{G. Militaru}
\address{Faculty of Mathematics and Computer Science, University of Bucharest, Str.
Academiei 14, RO-010014 Bucharest 1, Romania}
\email{gigel.militaru@fmi.unibuc.ro and gigel.militaru@gmail.com}

\thanks{A.L. Agore is Postdoctoral Fellow of the Fund for Scientific Research Flanders (Belgium) (F.W.O.– Vlaanderen). This work
was supported by a grant of the Romanian National Authority for
Scientific Research, CNCS-UEFISCDI, grant no. 88/05.10.2011.}



\maketitle

\begin{center}
\emph{Dedicated to Professor Constantin N\u{a}st\u{a}sescu on the
occasion of his 70th birthday}
\end{center}

\begin{abstract} Let $A \leq G$ be a subgroup of a group
$G$. An $A$-complement of $G$ is a subgroup $H$ of $G$ such that
$G = A H$ and $A \cap H = \{ 1\}$. The \emph{classifying
complements problem} asks for the description and classification
of all $A$-complements of $G$. We shall give the answer to this
problem in three steps. Let $H$ be a given $A$-complement of $G$
and $(\triangleright, \triangleleft)$ the canonical left/right
actions associated to the factorization $G = A H$. To start with,
$H$ is deformed to a new $A$-complement of $G$, denoted by $H_r$,
using a certain map $r: H \to A$ called a deformation map of the
matched pair $(A, H, \triangleright, \triangleleft)$. Then the
description of all complements is given: ${\mathbb H}$ is an
$A$-complement of $G$ if and only if ${\mathbb H}$ is isomorphic
to $H_{r}$, for some deformation map $r: H \to A$. Finally, the
classification of complements proves that there exists a bijection
between the isomorphism classes of all $A$-complements of $G$ and
a cohomological object ${\mathcal D} \, (H, A \, | \,
(\triangleright, \triangleleft) )$. As an application we show that
the theoretical formula for computing the number of isomorphism
types of all groups of order $n$ arises only from the
factorization $S_n = S_{n-1} C_n$.
\end{abstract}

\section*{Introduction}

Group factorizations have been intensively studied starting with
the classical papers by Sz\'{e}p \cite{szep1, szep2, szep3},
Douglas \cite{Douglas} and Ito \cite{Ito} but the problem goes
back to Maillet \cite{Maillet} and the 1900 Minkowski conjecture
on tiling (another name for factorizations) proved 40 years later
by Haj\'{o}s \cite{hajos}. Let $A \leq G$ be a subgroup of $G$. An
$A$-complement of $G$ is a subgroup $H \leq G$ such that $G$
factorizes through $A$ and $H$, that is $G = A H$ and $A \cap H =
\{ 1\}$. ${\mathcal F} (A, G)$ will denote the (possibly empty)
set of isomorphism types of all $A$-complements of $G$. We define
the factorization index of $A$ in $G$ to be the cardinal of
${\mathcal F} (A, G)$ and it will be denoted by $[G : A]^f := |\,
{\mathcal F} (A, G) \,|$.

The problem of existence of complements has to be treated ''case
by case'' for every given subgroup $A$ of $G$, a computational
part of it can not be avoided. It was studied in its global form:
\emph{find all factorizations of a given group $G$}. Particular
attention was given to finding all factorizations of simple
groups. Starting with the 1970's a very rich literature on the
subject was developed: see for instance \cite{aradf84, bozo},
\cite{fisman, fismanA, gentchev, Gi, goren}, \cite{LPS3},
\cite{Pr}, \cite{walls, WW}. For more details on this problem we
refer to the two fundamental monographs \cite{LPS1}, \cite{LPS2}
and the references therein. The present paper deals with the
following question:

\emph{\textbf{Classifying complements problem (CCP)}: Let $A$ be a
subgroup of $G$. If an $A$-complement of $G$ exists, describe
explicitly, classify all $A$-complements of $G$ and compute the
factorization index $[G : A]^f$ .}

We shall give the answer to the CCP in three steps called:
deformation of complements, description of complements and
classification of complements. First of all, in \seref{1} we
recall briefly the definition of a matched pair of groups and the
construction of the bicrossed product of two groups as defined by
Takeuchi \cite{Takeuchi}. Let $H$ be a given $A$-complement of $G$
and $(\triangleright, \triangleleft)$ the canonical left/right
actions associated to the factorization $G = A H$ such that $(A,
H, \triangleright, \triangleleft)$ is a matched pair of groups and
$G = A \bowtie H$. \thref{deformatANM} is called the
\emph{deformation of complements}: if $r : H \to A$ is a
deformation map of the matched pair $(A, H, \triangleright,
\triangleleft)$, then the group $H$ is deformed to a new group
$H_r$, called the $r$-deformation of $H$, such that $H_{r}$
remains an $A$-complement of $G = A \bowtie H$. The key point is
\thref{descformelorgr} called the \emph{description of
complements}: ${\mathbb H}$ is an $A$-complement of $G$ if and
only if ${\mathbb H}$ is isomorphic to $H_{r}$, for some
deformation map $r: H \to A$ of the canonical matched pair $(A, H,
\triangleright, \triangleleft)$. Finally, the \emph{classification
of complements} is proven in \thref{clasformelorgr}: there exists
a bijection between the set of isomorphism types of all
$A$-complements of $G$ and a cohomological type object ${\mathcal
D} \, (H, A \, | \, (\triangleright, \triangleleft) )$ which is
explicitly constructed. In particular, the factorization index is
computed by the formula $[G : A]^f = | \, {\mathcal D} (H, A \, |
\, (\triangleright, \triangleleft) ) \, |$. In \seref{fingr} we
provide some explicit examples. Let $S_{n}$ be the symmetric group
and $C_n$ the cyclic group of order $n$. By applying our results
to the factorization $S_n = S_{n-1} C_n$ we obtain the following:
$(1)$ any group $H$ of order $n$ is isomorphic to $(C_n)_r$, the
$r$-deformation of the cyclic group $C_n$ for some deformation map
$r: C_n \to S_{n-1}$ of the canonical matched pair $(S_{n-1}, C_n,
\triangleright, \triangleleft)$ and $(2)$ the number of
isomorphism types of all groups of order $n$ is equal to $| \,
{\mathcal D} (C_n, S_{n-1} \, | \, (\triangleright, \triangleleft)
) \, |$. Therefore, we obtain a combinatorial formula for
computing the number of isomorphism types of all groups of order
$n$ which arises from a minimal set of data: the factorization
$S_n = S_{n-1} C_n$.

The factorization problem as well as the bicrossed product were
introduced and studied in other fields such as topological groups,
local compact groups, Hopf algebras, groups and Lie algebras etc.
The results presented here for groups can be used as a model for
developing similar theories in the fields listed above. For Hopf
algebras and Lie algebras we refer to \cite{ABM} and for
associative algebras to \cite{AAL}.

\section{Preliminaries}\selabel{1}
Let $G$, $G'$ be two groups containing $A$ as a subgroup. We say
that a morphism of groups $\psi : G \to G' $ \emph{stabilizes} $A$
if $\psi (a) = a$, for all $a\in A$. Let $A$ and $H$ be two groups
and $\triangleright  : H \times A \rightarrow A$ and
$\triangleleft : H \times A \rightarrow H$ two maps. The map
$\triangleright $ (resp. $\triangleleft $) is called trivial if
$h\triangleright a = a$ (resp. $h\triangleleft a = h$), for all
$a\in A$ and $h\in H$. A \textit{matched pair} \cite{Takeuchi} of
groups is a quadruple $(A, H, \triangleright, \triangleleft)$,
where $A$ and $H$ are groups, $\triangleright : H \times A
\rightarrow A$ is a left action of the group $H$ on the set $A$,
$\triangleleft : H \times A \rightarrow H$ is a right action of
the group $A$ on the set $H$ satisfying the following
compatibilities for any $a$, $b \in A$, $h$, $g \in H$:
\begin{eqnarray}
h \triangleright (ab) &=& (h\triangleright a) ((h\triangleleft {a}
)\triangleright b) \eqlabel{2} \\
(hg)\triangleleft a &=& (h\triangleleft ({g \triangleright a}))
(g\triangleleft a) \eqlabel{3}
\end{eqnarray}

If $(A, H, \triangleright, \triangleleft)$ is a matched pair then
the following normalizing conditions hold:
\begin{equation}\eqlabel{mp1}
1 \triangleright a = a, \quad h \triangleleft 1 = h, \quad
h\triangleright 1 = 1, \quad 1\triangleleft a = 1
\end{equation}
for all $a \in A$ and $h\in H$. Let $\triangleright : H \times A
\rightarrow A$, $\triangleleft : H \times A \rightarrow H$ be two
maps and $A\bowtie \, H : = A\times H$ with the binary operation
defined by the formula:
\begin{equation}\eqlabel{def4}
(a,\, h)\cdot (b, \, g) : = \bigl( a (h\triangleright b), \, (h
\triangleleft b) g \bigl)
\end{equation}
for all $a$, $b \in A$, $h$, $g \in H$. The following is
\cite[Proposition 2.2.]{Takeuchi}:

\begin{proposition}\prlabel{matched}
Let $A$ and $H$ be groups and $\triangleright : H \times A
\rightarrow A$, $\triangleleft : H \times A \rightarrow H$ two
maps. Then $A\bowtie \, H$  is a group with unit $(1,1)$ if and
only if $(A, H, \triangleright, \triangleleft)$ is a matched pair
of groups. In this case $A\bowtie \, H$ is called the bicrossed
product of $A$ and $H$.
\end{proposition}

If $A\bowtie \, H$ is a bicrossed product then $i_A : A \to
A\bowtie \, H$, $i_A (a) = (a, 1)$ and $i_H : H \to A\bowtie \,
H$, $i_H (h) = (1, h)$ are morphisms of groups. $A$ and $H$ will
be viewed as subgroups of $A\bowtie \, H$ via the identifications
$A \cong A \times \{1\}$, $H \cong \{1\} \times H$. If the right
action $\triangleleft$ of a matched pair $(A, H, \triangleright,
\triangleleft)$ is the trivial action then the bicrossed product
$A\bowtie \, H$ is just the semidirect product $A\ltimes H$ of $A$
and $H$. Thus, the bicrossed product is a generalization of the
semidirect product to the case when none of the factors is
required to be normal.

We recall that a group $G$ \textit{factorizes} through two
subgroups $A$ and $H$ if $G = AH$ and $A \cap H = \{1\}$. The
bicrossed product $A\bowtie \, H$ factorizes through $A \cong A
\times \{1\}$ and $H \cong \{1\} \times H$ as for any $a\in A$ and
$h\in H$ we have that $ (a, h) = (a , 1) \cdot (1, h)$.
Conversely, the main motivation for defining the  bicrossed
product of groups is the following:

\begin{proposition}\prlabel{imp}
A group $G$ factorizes through two subgroups $A$ and $H$ if and
only if there exists a matched pair of groups $(A, H,
\triangleright, \triangleleft)$ such that the multiplication map
$$
m_{G}: A \bowtie H \to G, \qquad m_{G}(a, \, h) = ah
$$
for all $a\in A$ and $h\in H$ is an isomorphism of groups that
stabilizes $A$.
\end{proposition}

\begin{proof} The detailed proof is given in
\cite[Proposition 2.4]{Takeuchi}. We only indicate the
construction of the matched pair $(A, H, \triangleright,
\triangleleft)$ associated to the factorization $G = A H$. Indeed,
if $G$ factorizes through $A$ and $H$ then for any $g \in G$ there
exists a unique pair $(a, h) \in A \times H$ such that $g = ah$.
This allows us to attach to any $(a, h) \in A \times H$ a unique
pair of elements $(h \triangleright a,\, h \triangleleft a) \in A
\times H$ such that
\begin{equation}\eqlabel{constr}
h\, a =  (h \triangleright a ) (h \triangleleft a) \in A H
\end{equation}
Then $(A, H, \triangleright, \triangleleft)$ is a matched pair of
groups and $m_{G}: A \bowtie H \to G$ is an isomorphism of groups
that stabilizes $A$.
\end{proof}

\begin{remark} \relabel{neunic}
Let $A \leq G$ be a given subgroup of $G$. We will see that a
factorization $G = A H$ is not necessarily unique as there may
exist other subgroups $H'\leq G$, not isomorphic to $H$, such that
$G = A H'$. Such an example is presented below. Let $k$ be a
positive integer. In what follows we view $A_{4k-1}$ as a subgroup
of $A_{4k}$ by letting $4k$ to be a fixed point in the alternating
group $A_{4k}$. Then we have two factorizations: $A_{4k} =
A_{4k-1} D_{4k} = A_{4k-1} (C_2 \times C_{2k})$, where $D_{4k}$ is
the dihedral group and $C_m$ is the cyclic group of order $m$.
Indeed, let $\sigma$, $\tau \in A_{4k}$ be the even permutations
\begin{eqnarray*}
\sigma &=& (1, 3, 5, \cdots , 4k-1)(2, 4, 6, \cdots, 4k )\\
\tau &=& (1, 2k+2) (2, 2k+1) (3, 2k+4) (4, 2k+3) \cdots (2k-1, 4k)
(2k, 4k-1)
\end{eqnarray*}
It is straightforward to check that $\sigma$ and $\tau$ generate a
subgroup of $A_{4k}$ isomorphic to the dihedral group $D_{4k}$ of
order $4k$ and $A_{4k} = A_{4k-1} D_{4k}$. On the other hand, let
$\sigma'$, $\tau' \in A_{4k}$ given by
$$
\sigma' = (1, 2, \cdots, 2k) (2k+1, 2k+2, \cdots, 4k), \quad \tau'
= (1, 2k+1) (2, 2k+2) \cdots (2k, 4k)
$$
Then $\sigma' \tau' = \tau' \sigma'$ and the subgroup of $A_{4k}$
generated by $\sigma$ and $\tau$ is $C_2 \times C_{2k}$. Moreover,
we have $A_{4k} = A_{4k-1} (C_2 \times C_{2k})$. This example
reveals yet another important fact: a possible attempt to
generalize the Krull-Schmidt decomposition of groups into direct
products (\cite[Theorem 6.36]{rotman}) fails for bicrossed
products since $A_{4k} = A_{4k-1} \bowtie D_{4k} \cong A_{4k-1}
\bowtie (C_2 \times C_{2k})$, and of course the direct product
$C_2 \times C_{2k}$ is not isomorphic to the dihedral group
$D_{4k}$.
\end{remark}

From now on, the matched pair constructed in \equref{constr} will
be called the \emph{canonical matched pair} associated to the
factorization $G = A H$. We use the above terminology in order to
distinguish this matched pair among other possible matched pairs
$(A, H, \triangleright', \triangleleft')$ such that $A \bowtie' H
\cong G$ (isomorphism of groups that stabilizes $A$), where $A
\bowtie' H$ is the bicrossed product associated to the matched
pair $(A, H, \triangleright', \triangleleft')$. The following
result provides more details: it can be obtained from
\cite[Proposition 2.1]{CENT} for $\sigma = Id_{H}$. However, we
state the result below for the sake of completeness as it will be
used in the sequel.

\begin{proposition}\prlabel{1}
Let $(A, H, \triangleright, \triangleleft)$ and $(A, H',
\triangleright', \triangleleft')$ be two matched pairs of groups.
There exists a bijection between the set of all morphisms of
groups $\psi: A \bowtie' H' \rightarrow A \bowtie H$ that
stabilize $A$ and the set of all pairs $(r, v)$, where $r: H'
\rightarrow A$, $v: H' \rightarrow H$ are two unit preserving maps
satisfying the following compatibilities for any $h'$, $g' \in
H'$, $a \in A$:
\begin{eqnarray}
h' \triangleright' a &{=}& r(h') \, \bigl(v(h')
\triangleright a \bigl) \, r (h' \triangleleft' a)^{-1}  \eqlabel{3ab}\\
v(h' \triangleleft' a) &{=}& v(h') \triangleleft a
\eqlabel{4ab}\\
r(h' g' ) &{=}& r(h') \,
\bigl(v(h') \triangleright r(g')\bigl)\eqlabel{1ab}\\
v(h' g') &{=}& \bigl(v(h') \triangleleft r(g')\bigl) \,
v(g')\eqlabel{2ab}
\end{eqnarray}
Under the above correspondence the morphism of groups $\psi: A
\bowtie' H' \rightarrow A \bowtie H$ corresponding to $(r, v)$ is
given by:
\begin{equation}\eqlabel{p5}
\psi(a, \, h') = \bigl(a \, r(h'), \, v(h')\bigl)
\end{equation}
for all $a \in A$, $h' \in H'$ and $\psi : A \bowtie' H'
\rightarrow A \bowtie H$ is an isomorphism of groups if and only
if the map $v: H' \to H$ is bijective.
\end{proposition}

\section{Classifying complements}\selabel{bdt}
This section contains the main results of the paper. First we need
to introduce the following:

\begin{definition}\delabel{bicrformsb}
Let $A \leq G$ be a subgroup of $G$. An \emph{$A$-complement of
$G$} is a subgroup $H \leq G$ such that $G$ factorizes through $A$
and $H$. We denote by ${\mathcal F} (A, G)$ the set of isomorphism
types of all $A$-complements of $G$. We define the
\emph{factorization index} of $A$ in $G$ as the cardinal of
${\mathcal F} (A, G)$ and it will be denoted by $[G : A]^f := |\,
{\mathcal F} (A, G) \,|$. We shall write $[G : A]^f = 0$, if
${\mathcal F} (A, G)$ is empty.
\end{definition}

Let $H$ be a given $A$-complement of $G$ and $(A, H,
\triangleright, \triangleleft)$ the canonical matched pair
associated to it as in \equref{constr} of \prref{imp}. We shall
describe all $A$-complements of $G$ in terms of $(H,
\triangleleft, \triangleright)$ and certain maps $r : H \to A$,
called deformation maps. The classification of all $A$-complements
of $G$ is also given by proving that ${\mathcal F} (A, G)$ is in
bijection with a cohomological object.

\begin{examples} \exlabel{exidex}
1. Many group extensions $A \leq G$ have the factorization index
$[G : A]^f$ equal to $0$ (that is there exists no factorization $G
= A H$) or $1$. For instance, if $G$ is an abelian group, then $[G
: A]^f \in \{0, 1\}$, for any subgroup $A$ of $G$ ($[G : A]^f = 1$
if and only if $A$ is a direct summand of $G$).

Group extensions $A \leq G$ of factorization index $1$ are exactly
those for which the factorization is unique. In other words, for
these extensions the Krull-Schmidt theorem \cite[Theorem
6.36]{rotman} for bicrossed products holds: if $G\cong A \bowtie H
\cong A \bowtie H'$, then $H \cong H'$. A generic example of an
extension of factorization index $1$ is provided in \coref{proddi}
below: if $A \ltimes H$ is an arbitrary semidirect product of $A$
and $H$, then $[A\ltimes H : A]^f = 1$.

2. Examples of extensions $A \leq G$ for which $[G : A]^f \geq 2$
are quite rare, which makes them tempting to identify.
\reref{neunic} proves in fact that $[A_{4k}: A_{4k-1}]^f \geq 2$.
We provide below an example of an extension of factorization index
$2$.

The extension $S_{3} \leq S_{4}$ has factorization index $2$.
Indeed, let $C_4 = <(1234)> $ be the cyclic group of order $4$ and
$ C_2\times C_2$ the Klein's group viewed as a subgroup of $S_4$
being generated by $(12) (34)$ and $(13)(24)$. Then $S_4$ has two
factorizations: $S_4 = S_3 C_4 = S_3 (C_2 \times C_2)$. Since
there are no other groups of order four we obtain that $[S_4 :
S_3]^f = 2$.

3. Example $(2)$ above can be generalized as follows: the
factorization index $[S_n : S_{n-1}]^f = g(n)$, the number of
isomorphism types of groups of order $n$. Indeed, let $H$ be a
group of order $n$. We see $H$ as a subgroup of $S_{n}$ through
the regular representation, i.e. $T: H \to S_{n}$ given by $T(h) =
\sigma_{h}$, where $\sigma_{h}(x) = hx$, for all $h$, $x \in H$.
It is now obvious that through this representation $n$ is not
fixed by any other element in $H$ besides $1$. Since we consider
$S_{n-1}$ as a subgroup in $S_{n}$ by letting $n$ to be a fixed
point we have $H \cap S_{n-1} = 1$ and therefore $S_{n} =
S_{n-1}H$.
\end{examples}

\begin{definition}\delabel{descmap}
Let $(A, H, \triangleright, \triangleleft)$ be a matched pair of
groups. A \emph{deformation map} of the matched pair $(A, H,
\triangleright, \triangleleft)$ is a function $r : H \to A$ such
that $r (1) = 1$ and for all $g$, $h\in H$ we have:
\begin{equation}\eqlabel{compdef}
r\bigl( \bigl( h \triangleleft r(g)\bigl) \, g \, \bigl) = r(h) \,
\bigl( h \triangleright r(g)\bigl)
\end{equation}
\end{definition}

Let ${\mathcal D} {\mathcal M} \,(H, A \, | \, (\triangleright,
\triangleleft) )$ be the set of all deformation maps of the
matched pair $(A, H, \triangleright, \triangleleft)$. The trivial
map $H \to A$, $h \mapsto 1$, for any $h\in H$ is a deformation
map. If both actions $(\triangleright, \triangleleft)$ of the
matched pair are trivial then a deformation map is just a morphism
of groups $r: H \to A$. The following result is called the
deformation of complements: it shows that any $A$-complement can
be deformed to a new $A$-complement using a deformation map $r : H
\to A$.

\begin{theorem}\thlabel{deformatANM}
Let $(A, H, \triangleright, \triangleleft)$ be a matched pair of
groups and $r : H \to A$ a deformation map. The following hold:

$(1)$ Let $H_{r} := H$, as a set, with the new multiplication
$\bullet$ on $H$ defined for any $h$, $g\in H$ as follows:
\begin{equation}\eqlabel{defoinmult}
h \, \bullet \, g  := \bigl( h \triangleleft r(g)\bigl) \, g
\end{equation}
Then $(H_{r}, \bullet)$ is a group called the $r$-deformation of
$H$.

$(2)$ The map
\begin{equation} \eqlabel{3inv}
\triangleright^r : H_r \times A \to A, \quad   h \triangleright^r
\, a := r(h) \, \bigl( h \triangleright a \bigl) \, r(h
\triangleleft a)^{-1}
\end{equation}
for all $h\in H_r$, $a\in A$ is a left action of the group $H_r$
on the set $A$ and $(A, H_{r}, \, \triangleright^r,
\triangleleft)$ is a matched pair of groups. Furthermore, the map
\begin{equation}\eqlabel{psi2b}
\psi : A \bowtie^r H_{r} \to A \bowtie H, \quad \psi(a, \,h) = (a
\, r(h), \, h)
\end{equation}
for all $a\in A$ and $h\in H$ is an isomorphism of groups, where
$A \bowtie^r H_{r}$ is the bicrossed product associated to the
matched pair $(A, H_{r}, \, \triangleright^r, \triangleleft)$.

$(3)$ $H_{r}$ is an $A$-complement of $A \bowtie H$.
\end{theorem}

\begin{proof} $(1)$ Using the normalizing conditions \equref{mp1}
and the fact that $r: H \to A$ is a unitary map, $1$ remains the
unit for the new multiplication $\bullet$ given by
\equref{defoinmult}. On the other hand for any $h$, $g$, $t \in H$
we have:
\begin{eqnarray*}
(h \, \bullet \, g) \, \bullet \, t &{=}& \bigl[\bigl(h
\triangleleft r(g)\bigl) g \bigl] \, \bullet \, t =
\Bigl(\underline{\bigl((h \triangleleft r(g)) g
\bigl) \triangleleft \, r(t)} \Bigl) \, t \\
&\stackrel{\equref{3}} {=}& \Bigl( \underline{\bigl( h
\triangleleft r(g)\bigl) \triangleleft \bigl( g \triangleright
r(t)\bigl)}\Bigl)
\bigl( g \triangleleft r(t)\bigl) \, t \\
&{=}& \Bigl( h \triangleleft \bigl( \underline{r(g) ( g
\triangleright r(t)})\bigl)\Bigl) \bigl( g
\triangleleft r(t)\bigl) \, t \\
&\stackrel{\equref{compdef}} {=}& \Bigl( h \triangleleft r \bigl((
g \triangleleft r(t)) \,t \bigl)\Bigl) \, \bigl(g \triangleleft
r(t)\bigl) \, t \\
&{=}& h \, \bullet \bigl[\bigl( g \triangleleft r(t)\bigl) \, t
\bigl] = h \, \bullet \, (g \, \bullet \, t)
\end{eqnarray*}
Thus, the multiplication $\bullet$ is associative and has $1$ as a
unit. We prove now that the inverse of an element $h \in H_{r}$ is
given by $ h^{-1} = h^{-1} \triangleleft r(h)^{-1}$, for all $h
\in H$. Indeed, for any $h\in H$ we have:
\begin{eqnarray*}
h^{-1} \bullet h &{=}& \bigl(h^{-1} \triangleleft r(h)^{-1}\bigl)
\, \bullet \, h = \Bigl(\bigl(h^{-1} \triangleleft r(h)^{-1}\bigl)
\triangleleft \, r(h)\Bigl) \, h \\
&{=}& \Bigl(h^{-1} \triangleleft \bigl(r(h)^{-1}r(h)\bigl)\Bigl)\,
h = h^{-1} \, h  = 1
\end{eqnarray*}
Thus we proved that $(H_{r}, \bullet)$ is a monoid in which every
element has a left inverse. Hence $(H_{r}, \bullet)$ is a group.

$(2)$ Instead of using a rather long computation to prove that
$(A, H_{r}, \, \triangleright^r, \triangleleft)$ satisfies the
axioms \equref{2}-\equref{3} of a matched pair we proceed as
follows: first, observe that the map $\psi : A \times H_{r} \to A
\bowtie H$, $\psi (a, \, h) = (a \, r(h), \, h)$ is a bijection
between the set $A \times H_{r}$ and the group $A \bowtie H$ with
the inverse given by
$$
\psi^{-1} : A \bowtie H \to A \times H_{r}, \quad \psi^{-1} (a, \,
h) = (a \,  r(h)^{-1}, \, h)
$$
for all $a\in A$ and $h\in H$. Thus, there exists a unique group
structure $\diamond$ on the set $A \times H_{r}$ such that $\psi$
becomes an isomorphism of groups and this unique group structure
$\diamond$ is obtained by transferring the group structure from
the group $A \bowtie H$ via the bijection of sets $\psi$, i.e. is
given by:
$$
(a, \, h) \diamond (b, \, g) := \psi^{-1} \bigl(\psi(a,\, h) \,
\cdot \,  \psi(b,\, g)\bigl)
$$
for all $a$, $b \in A$ and $h$, $g \in H_r = H$. If we prove that
this group structure $\diamond$ on the direct product of sets $A
\times H_{r}$ is exactly the one given by \equref{def4} associated
to the pair of maps $(\triangleright^r, \triangleleft)$ the proof
is finished by using \prref{matched}. Indeed, for any $a$, $b\in
A$ and $g$, $h\in H$ we have:
\begin{eqnarray*}
(a, \, h) \diamond (b, \, g) &{=}& \psi^{-1} \bigl(\psi(a, \, h)
\, \cdot \, \psi(b, \, g)\bigl) = \psi^{-1} \Bigl(\bigl( a \,
r(h), \, h \bigl) \, \cdot \, \bigl( b \, r(g), \, g \bigl)\Bigl) \\
&{=}& \psi^{-1} \Bigl( a \, r(h) \bigl( h \triangleright b r(g)
\bigl), \,
\bigl( h \triangleleft b r(g)\bigl) \, g \Bigl)\\
&{=}& \Bigl(a \, r(h) \bigl( h \triangleright b r(g) \bigl) r
\Bigl( \bigl( h \triangleleft b r(g) \bigl) \, g \Bigl)^{-1}, \,
\bigl( h \triangleleft b r(g)\bigl) \, g\Bigl) \\
&{=}& \Bigl(a \,
r(h) \bigl( h \triangleright b r(g)\bigl) \underline{r \Bigl(
\bigl(( h \triangleleft b) \triangleleft r(g) \bigl) \, g
\Bigl)^{-1}}, \,
\bigl( h \triangleleft b r(g) \bigl) \, g\Bigl)\\
&\stackrel{\equref{compdef}} {=}& \Bigl(a \, r(h)
\bigl(\underline{ h \triangleright b r(g)} \bigl)\Bigl[ r \bigl( h
\triangleleft b\bigl) \Bigl( \bigl( h \triangleleft b \bigl)
\triangleright r(g)\Bigl) \Bigl]^{-1}, \, \bigl( h \triangleleft b
r(g)\bigl) \, g \Bigl)\\
&\stackrel{\equref{2}} {=}& \Bigl(a \,
r(h) \bigl( h \triangleright b \bigl) \underline{\Bigl( \bigl( h
\triangleleft b\bigl) \triangleright r(g) \Bigl) \Bigl(\bigl( h
\triangleleft b \bigl)
\triangleright r(g) \Bigl)^{-1}}r \bigl( h \triangleleft b \bigl)^{-1},\\
&&\bigl( h \triangleleft b r(g)\bigl) \, g\Bigl) \\
&{=}& \Bigl(a \, r(h) \bigl( h \triangleright b \bigl)r \bigl( h
\triangleleft b \bigl)^{-1},\, \underline{\bigl( h \triangleleft b
r(g)\bigl) \, g}\Bigl)\\
&\stackrel{\equref{defoinmult}}{=}& \Bigl(\underline{a \, r(h)
\bigl( h \triangleright b \bigl)r \bigl( h \triangleleft b
\bigl)^{-1}},\,\bigl( h \triangleleft b
\bigl) \bullet g\Bigl)\\
&\stackrel{\equref{3inv}}{=}& \Bigl(a \, (h \triangleright^r b),
\, (h \triangleleft b) \bullet g \bigl) = (a, \, h) \cdot^r (b, \,
g)
\end{eqnarray*}
where $\cdot^r$ is the multiplication given by \equref{def4}
associated to the new pair of maps $(\triangleright^r,
\triangleleft)$. Now we apply \prref{matched}.

$(3)$ First we remark that the isomorphism of groups $\psi : A
\bowtie^r H_{r} \to A \bowtie H$ given by \equref{psi2b}
stabilizes $A$. Hence $A \cong \psi (A) = A \times \{1\} \leq A
\bowtie H$ and $H_r \cong \psi( \{1\} \times H_r) = \{ (r(h), h)
\, |\,  h\in H \}$ is a subgroup of $A \bowtie H$. Now, $A \bowtie
H$ factorizes through $A$ and $H_r$ since in $A \bowtie H$ we
have:
$$
(a, h) = (a r(h)^{-1}, 1) \cdot (r(h), h)
$$
for all $a\in A$ and $h\in H$. Of course, $A \times \{1\}$ and $\{
(r(h), h) \, |\,  h\in H \} \cong H_r$ have trivial intersection
in $A \bowtie H$ as $r$ is a unitary map. The proof is now
completely finished.
\end{proof}

Now we prove the converse of \thref{deformatANM} which gives the
description of all $A$-complements of $G$ in terms of a
fixed one $H$.

\begin{theorem}\thlabel{descformelorgr}
Let $A \leq G$ be a subgroup of $G$ and $H$ a given
$A$-complement of $G$. Then ${\mathbb H}$ is an $A$-complement of
$G$ if and only if there exists an isomorphism of groups $
{\mathbb H} \cong H_{r}$, for some deformation map $r : H \to A$ of
the canonical matched pair $(A, H, \triangleright, \triangleleft)$
associated to the factorization $G = A H$.
\end{theorem}

\begin{proof}
Let $A \bowtie H$ be the bicrossed product of the canonical
matched pair $(A, H, \triangleright, \triangleleft)$. Then the
multiplication map $m_G: A \bowtie H \to G$ is an isomorphism of
groups that stabilizes $A$. Consider $(A, {\mathbb H},
\triangleright', \triangleleft')$ to be the canonical matched pair
associated to the factorization $G = A {\mathbb H}$; hence the
multiplication map $m_G' : A \bowtie' {\mathbb H} \to G$ is also
an isomorphism of groups that stabilizes $A$. Then $\psi : =
m_G^{-1} \circ m_G': A \bowtie' {\mathbb H} \to A \bowtie H$ is a
group isomorphism that stabilizes $A$ as a composition of such
morphisms. Now by applying \prref{1} it follows that $\psi$ is
uniquely determined by a pair of maps $(\overline{r},
\overline{v})$ consisting of a unitary map $\overline{r}: {\mathbb
H} \to A$ and a unitary bijective map $\overline{v} : {\mathbb H}
\to H $ satisfying the compatibility conditions
\begin{eqnarray}
h' \triangleright' a &{=}& \overline{r}(h') \,
\bigl(\overline{v}(h')
\triangleright a \bigl) \, \overline{r} (h' \triangleleft' a)^{-1}  \eqlabel{3ab}\\
\overline{v}(h' \triangleleft' a) &{=}& \overline{v}(h')
\triangleleft a
\eqlabel{4ab}\\
\overline{r}(h' g' ) &{=}& \overline{r}(h') \,
\bigl(\overline{v}(h') \triangleright \overline{r}(g')\bigl)\eqlabel{1ab}\\
\overline{v}(h' g') &{=}& \bigl(\overline{v}(h') \triangleleft
\overline{r}(g')\bigl) \, \overline{v}(g')\eqlabel{2ab}
\end{eqnarray}
for all $h'$, $g' \in {\mathbb H}$ and $a\in A$. Moreover, $\psi :
A \bowtie' {\mathbb H} \to A \bowtie H$ is given by:
\begin{equation*}\eqlabel{psi2aab}
\psi (a, \, h') = (a \, \overline{r}(h'), \, \overline{v}(h'))
\end{equation*}
for all $a\in A$ and $h'\in {\mathbb H}$. We define
$$
r : H \to A, \quad r := \overline{r} \circ \overline{v}^{-1}
$$
and we will prove that $r$ is a deformation map of the matched
pair $(A, H, \triangleright, \triangleleft)$ and $\overline{v}:
{\mathbb H} \to H_{r}$ is an isomorphism of groups. First, notice
that $r $ is unitary as $\overline{r}$, $\overline{v}$ are both
unitary. We have to show that the compatibility condition
\equref{compdef} holds for $r$. Indeed, from \equref{1ab} and
\equref{2ab} we obtain:
\begin{eqnarray}\eqlabel{rv}
\overline{r} \circ \overline{v}^{-1} [ \, \bigl(\overline{v}(h')
\triangleleft \overline{r}(g')\bigl) \, \overline{v}(g') \, ] =
\overline{r}(h')\bigl(\overline{v}(h')
\triangleright\overline{r}(g')\bigl)
\end{eqnarray}
for all $h'$, $g' \in {\mathbb H}$. Let $h$, $g \in H$ and write
the compatibility condition \equref{rv} for $h' =
\overline{v}^{-1} (h)$ and $g' = \overline{v}^{-1} (g)$. We obtain
$$
r\Bigl( \bigl( h \triangleleft r(g)\bigl) \, g \, \Bigl) = r(h) \,
\bigl( h \triangleright r(g)\bigl)
$$
that is \equref{compdef} holds and hence $r: H \to A$ is a deformation
map. Finally, $\overline{v}: {\mathbb H} \to H_{r}$ is a bijective
map as $H = H_r$ as sets. Hence, we are left to prove that
$\overline{v}$ is also a morphism of groups. Indeed, for any $h'$,
$g' \in {\mathbb H}$ we have:
\begin{eqnarray*}
\overline{v}(h'g') \stackrel{\equref{2ab}}= \bigl(\overline{v}(h')
\triangleleft \overline{r}(g')\bigl) \, \overline{v}(g')
\stackrel{\equref{defoinmult}}= \overline{v}(h') \bullet
\overline{v}(g')
\end{eqnarray*}
where $\bullet$ is the multiplication on $H_{r}$ as defined by
\equref{defoinmult}. Hence $\overline{v}: {\mathbb H} \to H_{r}$
is an isomorphism of groups and the proof is finished.
\end{proof}

\begin{remark} \relabel{remarcatri}
Assume that in \thref{deformatANM} the deformation map $r: H \to
A$ is the trivial one or the right action $\triangleleft$ is the
trivial action of $A$ on $H$. Then $H_{r} = H$ as groups. In
general, the new group $H_{r}$ may not be isomorphic to $H$ as
groups. \exref{primexem} shows how the Klein's group $C_2 \times
C_2$ can be constructed as an $r$-deformation of the cyclic group
$C_4$, for some deformation map $r : C_4 \to S_3$. On the other
hand, there are also examples of non-trivial deformation maps,
with a non-trivial action $\triangleleft$, such that $H_{r}$ is a
group isomorphic to $H$. Such an example is provided in
\exref{treiexemp}.
\end{remark}

\begin{corollary} \colabel{proddi}
Let $A$ and $H$ be two groups, $A \ltimes H$ an arbitrary
semidirect product of $A$ and $H$. Then the factorization index
$[A\ltimes H : A]^f = 1$.

In particular, the following Krull-Schmidt type theorem for
bicrossed product holds: if $A \ltimes H \cong A \bowtie H'$
(isomorphism of groups that stabilizes $A$), then the groups $H'$
and $H$ are isomorphic, where $A \bowtie H'$ is an arbitrary
bicrossed product.
\end{corollary}

\begin{proof}
Indeed, $H \cong \{1\} \times H$ is an $A$-complement of
the semidirect product $A\ltimes H$. Moreover, the right action
$\triangleleft $ of the canonical matched pair $(A, H,
\triangleright, \triangleleft)$ constructed in \equref{constr} for
the factorization $A \ltimes H = (A \times \{1\}) ( \{1\} \times
H)$ is the trivial action. Thus, using \reref{remarcatri}, any
$r$-deformation of $H\cong \{1\} \times H$ coincides with $H$. The
rest follows from \thref{descformelorgr}.
\end{proof}

In order to provide the classification of complements we need one more definition:

\begin{definition}\delabel{amisur2}
Let $(A, H, \triangleright, \triangleleft)$ be a matched pair of
groups. Two deformation maps $r$, $R: H \to A$ are called
\emph{equivalent} and we denote this by $r \sim R$ if there exists
$\sigma: H \to H$ a permutation on the set $H$ such that $\sigma
(1_H) = 1_H$ and for all $g$, $h\in H$ we have:
\begin{equation}\eqlabel{echivamit}
\sigma \bigl( (h \triangleleft r(g)) \, g \bigl) = \bigl( \sigma
(h) \triangleleft R (\sigma (g) )  \bigl) \, \sigma (g)
\end{equation}
\end{definition}

As a conclusion of all the above results, our main theorem which
gives the classification of all $A$-complements of a group
$G$ now follows.

\begin{theorem}\thlabel{clasformelorgr}
Let $A \leq G$ be a subgroup of $G$, $H$ a given
$A$-complement of $G$ and $(A, H, \triangleright, \triangleleft)$ the
associated canonical matched pair. Then:

$(1)$ $\sim$ is an equivalence relation on ${\mathcal D}{\mathcal
M} (H, A \, | \, (\triangleright, \triangleleft) )$ and the map
$$
{\mathcal D} \, (H, A \, | \, (\triangleright, \triangleleft) ) \,
\to {\mathcal F} \, (A, G), \quad \overline{r} \mapsto H_{r}
$$
is a bijection between sets, where ${\mathcal D} \, (H, A \, | \,
(\triangleright, \triangleleft) ) := {\mathcal D}{\mathcal M} \,
(H, A \, | \, (\triangleright, \triangleleft) )/\sim$ is the
quotient set through the relation $\sim$ and $\overline{r}$ is the
equivalence class of $r$ via $\sim$.

$(2)$ The factorization index $[G : A]^f$ is computed by the
formula:
$$
[G : A]^f = | {\mathcal D} \, (H, A \, | \, (\triangleright,
\triangleleft) )|
$$
\end{theorem}

\begin{proof}
It follows from \thref{descformelorgr} that if ${\mathbb H}$ is an
arbitrary $A$-complement of $G$, then there exists an isomorphism
of groups ${\mathbb H} \cong H_{r}$, for some deformation map $r :
H \to A$ of the matched pair $(A, H, \triangleright,
\triangleleft)$. Thus, in order to classify all $A$-complements on
$G$ we can consider only $r$-deformations of $H$, for various
deformation maps $r : H \to A$. Now let $r$, $R : H \to A$ be two
deformation maps of the matched pair $(A, H, \triangleright,
\triangleleft)$. As $H_r$ and $H_R$ coincide as sets, the groups
$H_r$ and $H_R$ are isomorphic if and only if there exists $\sigma
: H \to H$ a unitary bijective map such that $\sigma : H_r \to
H_R$ is a morphism of groups. Taking into account the definition
of the multiplication on $H_r$ given by \equref{defoinmult} it
follows that $\sigma$ is a group morphism if and only if the
compatibility condition \equref{echivamit} of \deref{amisur2}
holds, i.e. $r \sim R$. Hence, $r \sim R$ if and only if there
exists a map $\sigma$ such that $\sigma: H_r \to H_R$ is an
isomorphism of groups. Therefore $\sim$ is an equivalence relation
on ${\mathcal D}{\mathcal M} (H, A \, | \, (\triangleright,
\triangleleft) )$ and the map
$$
{\mathcal D} \, (H, A \, | \, (\triangleright, \triangleleft) )
\to {\mathcal F} \, (A, E),  \qquad \overline{r} \mapsto H_{r}
$$
is well defined and a bijection between sets, where $\overline{r}$
is the equivalence class of $r$ via the relation $\sim$. $(2)$
follows from $(1)$ and the proof is now finished.
\end{proof}

\section{Examples}\selabel{fingr}

Let $n$ be a positive integer. In this section we apply the
results obtained in \seref{bdt} to the factorization $S_{n} =
S_{n-1} C_{n}$. As a consequence, we derive a combinatorial
formula for computing the number of types of groups of order $n$
as well as an explicit description for the multiplication on any
group of order $n$. In what follows we consider the usual
presentation of the symmetric group $S_{n}$:
$$
S_{n} = \langle s_{1}, \, s_{2}, \, \ldots, \, s_{n-1} ~|~
s_{i}^{2} = 1, \, s_{i} \, s_{i+1} \, s_{i} = s_{i+1} \, s_{i} \,
s_{i+1}, \, s_{i} \, s_{j} = s_{j} \, s_{i}, \, |i-j|>1 \rangle
$$
We shall see the cyclic group $C_{n}$ as a subgroup of $S_{n}$
generated by $x := s_{1} \, s_{2} \, \ldots \, s_{n-1}$ while
$S_{n-1}$ will be seen as the subgroup of $S_{n}$ generated by
$s_{1}, \, s_{2}, \, \ldots \, s_{n-2}$. To start with, we
describe the canonical matched pair associated to the
factorization $S_{n} = S_{n-1} C_{n}$. It is enough to define the
two actions $\triangleright  : C_n \times S_{n-1} \rightarrow
S_{n-1}$ and $\triangleleft : C_n \times S_{n-1} \rightarrow C_n$
on the generators of $S_{n-1}$ and $C_{n}$ as they can be extended
to the entire group by using the compatibilities \equref{2} and
\equref{3}.

\begin{proposition}\prlabel{matchn}
The canonical matched pair $( S_{n-1}, C_n, \triangleright,
\triangleleft )$ associated to the factorization $S_{n} = S_{n-1}
C_{n}$ is given as follows:
\begin{eqnarray*}
x \triangleright s_{i} &=& \left\{
\begin{array}{rcl} s_{i+1}, & \mbox{if} & i < n-2\\ s_{n-2} \, s_{n-3} \, \ldots \, s_{1}, & \mbox{if}& i = n-2 \end{array}
\right. \\
x \triangleleft s_{i} &=& \left\{
\begin{array}{rcl} x, & \mbox{if} & i < n-2\\ x^{2}, & \mbox{if}& i = n-2 \end{array}\right.
\end{eqnarray*}
where $x := s_{1} \, s_{2} \, \ldots \, s_{n-1}$.
\end{proposition}

\begin{proof}
We compute the canonical matched pair by using the approach
highlighted in the proof of \prref{imp}. We start by computing the
$x s_{i}$'s, for al $i \in 1, 2, \ldots, n-2$. If $i < n-2$ we
have:
\begin{eqnarray*}
x s_{i} &=& s_{1} \ldots s_{i-1} \, s_{i} \, s_{i+1} \, s_{i+2}
\ldots
s_{n-1} \, s_{i}\\
&=& s_{1} \ldots s_{i-1} \,(s_{i} \, s_{i+1} \, s_{i})
\, s_{i+2} \ldots s_{n-1}\\
&=& s_{1} \ldots s_{i-1} \,(s_{i+1} \, s_{i} \, s_{i+1})
\, s_{i+2} \ldots s_{n-1}\\
&=& s_{i+1}\,s_{1} \ldots \, s_{n-1} = s_{i+1} x
\end{eqnarray*}
If $i = n-2$ we obtain:
\begin{eqnarray*}
x s_{n-2} &=& s_{1} \ldots s_{n-3} \, (s_{n-2} \, s_{n-1} \, s_{n-2})\\
&=& s_{1} \ldots s_{n-3} \, (s_{n-1} \, s_{n-2} \, s_{n-1})\\
&=& s_{n-1} \, s_{1} \ldots s_{n-1}\\
&=& s_{n-2} \, s_{n-3} \ldots s_{1} (s_{1} \, s_{2} \ldots
s_{n-1})^{2} = x' x^{2}
\end{eqnarray*}
where $x' := s_{n-2} \, s_{n-3} \, \ldots \, s_{1}$ and the
conclusion follows easily.
\end{proof}

By applying \thref{descformelorgr} and \thref{clasformelorgr} for
the factorization $S_n = S_{n-1} C_n$ we obtain the following
result concerning the structure and the number of types of groups
of finite order.

\begin{corollary}\colabel{cobobita}
Let $n$ be a positive integer and $(S_{n-1}, C_n, \triangleright,
\triangleleft)$ the canonical matched pair associated to the
factorization $S_n = S_{n-1} C_n$. Then:

$(1)$ Any group of order $n$ is isomorphic to an $r$-deformation
of the cyclic group $C_n$, for some deformation map $r: C_n \to
S_{n-1}$ of the canonical matched pair $(S_{n-1}, C_n,
\triangleright, \triangleleft)$. The multiplication $\bullet$ on
$(C_{n})_{r}$ is given by: $x \bullet y = \bigl( x \triangleleft
r(y) \bigl) \, y$, for all $x$, $y \in (C_{n})_{r}$, where we
denoted by juxtaposition the multiplication in the cyclic group
$C_{n}$.

$(2)$ The number of isomorphism types of all groups of order $n$
is equal to
$$
| \, {\mathcal D} (C_n, S_{n-1} \, | \, (\triangleright,
\triangleleft) ) \, |
$$
\end{corollary}

\begin{proof}
It follows from \thref{descformelorgr} and \thref{clasformelorgr}
taking into account that any group $H$ of order $n$ is an
$S_{n-1}$-complement of $S_n$ according to $(3)$ of \exref{exidex}.
\end{proof}

Now we provide some explicit examples in order to see how
\coref{cobobita} works.

\begin{example}\exlabel{primexem}
Consider the extension $S_3 \leq S_4$ of factorization index $2$.
Then the canonical matched pair $(S_{3}, C_4, \triangleright,
\triangleleft)$ associated to the factorization $S_4 = S_3 C_4$
from \prref{matchn} takes the following form:
$$
\begin{tabular} {c | c  c  c  c  c  c }
$\triangleright$ & $1$ & $s_{1}$ & $s_{1}\,s_{2}$ & $s_{2}\,s_{1}$ & $s_{2}$ & $s_{1}\,s_{2}\,s_{1}$\\
\hline $1$ & $1$ & $s_{1}$ & $s_{1}\,s_{2}$ & $s_{2}\,s_{1}$ & $s_{2}$ & $s_{1}\,s_{2}\,s_{1}$\\
$x$ & $1$ & $s_{2}$ & $s_{1}$ & $s_{1}\,s_{2}$ & $s_{2}\,s_{1}$ & $s_{1}\,s_{2}\,s_{1}$\\
$x^{2}$ & $1$ & $s_{2}\,s_{1}$ & $s_{2}$ & $s_{1}$ & $s_{1}\,s_{2}$ & $s_{1}\,s_{2}\,s_{1}$\\
$x^{3}$ & $1$ & $s_{1}\,s_{2}$ & $s_{2}\,s_{1}$ & $s_{2}$ &
$s_{1}$ & $s_{1}\,s_{2}\,s_{1}$
\end{tabular}
\quad
\begin{tabular} {c | c  c  c  c  c  c }
$\triangleleft$ & $1$ & $s_{1}$ & $s_{1}\,s_{2}$ & $s_{2}\,s_{1}$ & $s_{2}$ & $s_{1}\,s_{2}\,s_{1}$\\
\hline $1$ & $1$ & $1$ & $1$ & $1$ & $1$ & $1$\\
$x$ & $x$ & $x$ & $x^{2}$ & $x^{3}$ & $x^{2}$ & $x^{3}$\\
$x^{2}$ & $x^{2}$ & $x^{3}$ & $x^{3}$ & $x$ & $x$ & $x^{2}$\\
$x^{3}$ & $x^{3}$ & $x^{2}$ & $x$ & $x^{2}$ & $x^{3}$ & $x$
\end{tabular}
$$
By a straightforward computation one can prove that there are two
deformation maps for the canonical matched pair $(S_{3}, C_4,
\triangleright, \triangleleft)$: namely the trivial one $r' :
C_{4} \to S_{3}$, $r' (c) = 1$, for any $c \in C_4$ and the map
given by
$$
r: C_{4} \to S_{3}, \qquad r(1) = r(x^{2})  = 1, \qquad r(x) =
r(x^{3}) = s_{1}\,s_{2}\,s_{1}
$$
We consider the following presentation of the Klein's group:
$C_2\times C_2 = \langle a = (12)(34), \, b = (13)(24)~|~ a^{2} =
b^{2} =1, \, ab = ba \rangle$. Then we can easily prove that the
map:
$$
\varphi: C_{2} \times C_{2} \to (C_{4})_{r}, \quad \varphi(1) = 1,
\qquad \varphi(a) = x, \qquad \varphi(b) = x^{2}, \qquad
\varphi(ab) = x^{3}
$$
is an isomorphism of groups, that is $C_{2} \times C_{2} \cong
(C_{4})_{r}$.
\end{example}

\coref{cobobita} proves that any finite group of order $n$ is
isomorphic to an $r$-deformation of the cyclic group $C_n$, for
some deformation map $r: C_n \to S_{n-1}$ of the canonical matched
pair associated to the factorization $S_n = S_{n-1} C_n$. The next
example shows how the symmetric group $S_3$ appears as an
$r$-deformation of the cyclic group $C_6$ arising from a given
matched pair $(C_{3}, C_{6}, \triangleright, \triangleleft)$.

\begin{example} \exlabel{doiexemp}
Let $C_{3} = \langle a ~|~ a^{3} = 1 \rangle$ and $C_{6} = \langle
b ~|~ b^{6} = 1 \rangle$ be the cyclic groups of order $3$
respectively $6$. As a special case of \cite[Proposition
4.2]{CENT} we have a matched pair of groups $(C_{3}, C_{6},
\triangleright, \triangleleft)$, where the actions
$(\triangleright, \triangleleft)$ on generators are defined by:
$$
b \triangleright a := a^2 \qquad b \triangleleft a := b^3
$$
By a rather long but straightforward computation it can be seen
that the map:
$$
r: C_{6} \to C_{3}, \qquad r(1) = r(b^{3}) = 1, \qquad r(b) =
r(b^{4}) = a^{2}, \qquad r(b^{2}) = r(b^{5}) = a
$$
is a deformation map of the matched pair $(C_{3}, C_{6},
\triangleright, \triangleleft)$ and $\varphi : S_{3} \to
(C_{6})_{r}$ given by:
$$
\varphi(1) = 1, \, \varphi(s_{1}) = b, \, \varphi(s_{1}\,s_{2}) =
b^{2}, \varphi(s_{2}\,s_{1}) = b^{4}, \, \varphi(s_{2}) = b^{5},
\, \varphi(s_{1}\,s_{2}\,s_{1}) = b^{3}
$$
is an isomorphism of groups. Hence $S_{3}$ is an $r$-deformation
of the cyclic group $C_{6}$.
\end{example}

Our last example provides a non-trivial deformation map $r : H \to
A$ such that $H_r \cong H$.

\begin{example} \exlabel{treiexemp}
Let $(C_{3}, C_{6}, \triangleright, \triangleleft)$ be the matched
pair of \exref{doiexemp}. Then the map
$$
R: C_{6} \to C_{3}, \quad R(1) = R (b^{2}) = R (b^{4}) = 1, \qquad
R (b) = R (b^{3}) = R (b^{5}) = a
$$
is also a deformation map of $(C_{3}, C_{6}, \triangleright,
\triangleleft)$. Then, one can easily check that $(C_{6})_{R}$ is
a group isomorphic to $C_{6}$.
\end{example}

\section*{Acknowledgment}
The authors are indebted to the referee for his valuable
suggestions as well as to Marian Deaconescu, Drago\c{s}
Fra\c{t}ila, Michael Giudici and Cheryl Praeger for their comments
on a preliminary version of this paper.

\end{document}